\documentclass{article}

\usepackage{amsmath,amssymb,amsthm,mathrsfs}
\usepackage[colorlinks,linkcolor=red,anchorcolor=blue,citecolor=red]{hyperref}
\usepackage{cite}

\usepackage[
  left=2.5cm,
  right=2.5cm,
  top=2.5cm,
  bottom=2cm,
  headheight=0cm,
  footskip=1.0cm,
  paperwidth=21cm,
  paperheight=29.7cm
  ]{geometry}
\allowdisplaybreaks[4] 

\def\D{\Delta}

\def\d{\mathop{}\!\mathrm{d}}

\numberwithin{equation}{section}
\newtheorem{theorem}{Theorem}[section]
\newtheorem{lemma}[theorem]{Lemma}

\newtheorem{corollary}[theorem]{Corollary}
\newtheorem{remark}[theorem]{Remark}
\newtheorem{definition}[theorem]{Definition}
\newtheorem{proposition}[theorem]{Proposition}

\begin{document}

\title{\bf Blow-up problems  for a parabolic equation  coupled  with superlinear source  and  local  linear boundary dissipation}

\author{\bf
Fenglong Sun \thanks{Corresponding author},
Yutai  Wang ,
Hongjian  Yin
}
\date{}
\maketitle

\footnotetext[1]{E-mail addresses: sunfenglong@qfnu.edu.cn
(F. Sun), yutaiwang1224@163.com (Y. Wang),
yinhongjian2022@163.com (H. Yin). }
\begin{center}
{\footnotesize \it  School of Mathematical Sciences, Qufu Normal
University, Qufu 273165, Shandong,\\ People's Republic of China}\\

\end{center}

\begin{abstract}
In this paper, we consider the finite time blow-up results for  a parabolic equation coupled   with superlinear source term and local linear boundary dissipation. Using a concavity argument, we derive the sufficient conditions for the solutions to  blow up in finite time. In particular, we obtain  the existence  of finite time blow-up solutions with arbitrary high initial energy.  We also derive the upper bound and lower bound of the blow up time.
\end{abstract}

\begin{quote}
\textbf{MR Subject Classification 2020: }35B44, 35K58
\end{quote}

\section{Introduction}
In this paper, we consider the following parabolic equation coupled with    superlinear source term and local linear boundary dissipation:
\begin{equation}\label{P1}
\left\{
\begin{array}{ll}
  u_t-\D u =|u|^{p-2}u, & x\in \Omega,\,t>0, \\
  u(x,t)=0, & x\in \Gamma_0,\,t> 0, \\
  \frac{\partial u}{\partial \nu}={- u_t}, & x\in \Gamma_1,\,t>0,\\
  u(x,0)=u_0(x), & x\in \Omega,
\end{array}
\right.
\end{equation}
where $p> 2$, $\Omega$ is a bounded open subset of $\Bbb{R}^n$ ($n\geq 1$) with $C^1$ boundary $\partial \Omega$ and $\nu$   denotes the unit
outward normal vector to $\partial \Omega$. Let $\{\Gamma_0,\Gamma_1\}$ be a partition of the  boundary $\partial\Omega$ such that
\[\partial \Omega=\Gamma_0 \cup \Gamma_1,\quad \overline{\Gamma_0}\cap \overline{\Gamma_1}=\emptyset,\]
where $\Gamma_0$ and $\Gamma_1$ are measurable over $\partial \Omega$,
endowed with $(n-1)$-dimensional surface measure $\sigma$. We assume  $\sigma\left(\Gamma_0\right)>0$ throughout this paper.

Problem \eqref{P1} can be used to describe a heat reaction-diffusion process which occurs inside a solid body $\Omega$ surrounded by a fluid, with contact $\Gamma_1$ and having an internal cavity with contact boundary $\Gamma_0$. The function $u=u(x,t)$ represents the temperature at point $x$ and time $t$. The quantity of heat produced by the reaction is proportional to a superlinear power of the temperature, i.e. $|u|^{p-2}u$ with $p>2$.  To avoid the internal explosion  inside  $\Omega$,   a refrigerating system   is introduced in the fluid.  The refrigerating  system works in such a way that the heat absorbed from the fluid is proportional to a power of the rate of change of the temperature,  which   can be  described by
\[   \frac{\partial u}{\partial \nu}= -|u_t|^{m-2} u_t,\quad    x\in \Gamma_1,\]
where $\dfrac{\partial u}{\partial \nu}$ represents the heat flux from $\Omega$ to the fluid.

Evolution  equations   with dynamical boundary conditions have been studied by many authors. We refer the readers to \cite{Hintermann-PRES89,Ioan-EJDE01,Chueshov-CPDE02,Cavalcanti-JDE07,Enzo-jde11,YangXin-JDE16,Vitillaro-2017ARMA,Vitillaro-JDE18,Vitillaro-DCDS21} and the references therein. In this paper, we mainly focus on the finite time blow-up of solutions for parabolic equations with dynamical boundary conditions. In \cite{Levine-JDE74}, using a certain concavity technique,  Levine and Payne obtained the finite time blow-up result for  the following problem:
\begin{equation}\label{P2}
\left\{
\begin{array}{ll}
  u_t-\D u =0, & x\in \Omega,\,t>0, \\
  \frac{\partial u}{\partial \nu}=f(u), & x\in \partial \Omega,\,t>0,\\
  u(x,0)=u_0(x), & x\in \Omega,
\end{array}
\right.
\end{equation}
provided the initial data $u_0$ satisfies
\[ \frac{1}{2}\int_\Omega |\nabla u_0(x)|^2 \d x-\oint_{\partial \Omega} \left( \int_0^{u_0(s)} f(z)\d z  \right)\d s<0. \]
In \cite{Levine-PAMS74}, similar results were derived for more general  classes of higher order equations.  In \cite{Levine-MMAS87}, using the potential well method, Levine and Smith obtained the finite time blow-up result for  problem \eqref{P2} provided the initial data $u_0$ satisfies
\[  \frac{1}{2}\int_\Omega |\nabla u_0(x)|^2 \d x-\oint_{\partial \Omega} \left( \int_0^{u_0(s)} f(z)\d z  \right)\d s<d,\quad \int_\Omega |\nabla u_0(x)|^2 \d x-\oint_{\partial \Omega} f\left(u_0(s)\right)u_0(s)\d s<0,\]
where $d$ is the potential well depth. In \cite{Vitillaro-PRES-05}, Vitillaro obtained  the local  and global existence  for the solutions of the following heat equation   with local nonlinear boundary damping and source terms:
\begin{equation*}
\left\{
\begin{array}{ll}
  u_t-\D u =0, & x\in \Omega,\,t>0, \\
 u(x,t)=0, & x\in \Gamma_0,\,t> 0, \\
  \frac{\partial u}{\partial \nu}= -|u_t|^{m-2} u_t+|u|^{p-2}u, & x\in \Gamma_1,\,t>0,\\
  u(x,0)=u_0(x), & x\in \Omega.
\end{array}
\right.
\end{equation*}
In \cite{VitillaroDCDS13},  Fiscella and Vitillaro considered problem   \eqref{P1} with  local nonlinear boundary dissipation (i.e. $m\geq 2$). Using the   monotonicity method of J.L. Lions and a contraction argument, they showed the results of local well-posedness.  They also  proved that the weak solution blows up in finite time provided
\[ m<m_0(p):=\frac{2(n+1)p-4(n-1)}{n(p-2)+4}  \]
and
\[J(u_0)=\frac{1}{2}\int_\Omega |\nabla u_0(x)|^2 \d x-\frac{1}{p}\int_\Omega |u_0(x)|^p \d x<d,\quad K(u_0)=\int_\Omega |\nabla u_0(x)|^2 \d x- \int_\Omega |u_0(x)|^p \d x<0. \]

It is natural to ask whether or not problem \eqref{P1}  admits finite time blow-up solutions  with arbitrary high initial energy, especially for the case of $J(u_0)\geq d$. The main purpose of this paper is to answer this question for the case of   linear boundary dissipation (i.e. $m=2$). We find two subsets $\mathcal{B}_1$ and $\mathcal{B}_2$ in the space $H_{\Gamma_0}^1(\Omega)$, which are invariant under the semi-flow associated with   problem  (\ref{P1}). Combing them with a concavity argument, we prove that the weak solution of problem  (\ref{P1}) blows up in finite time provided the initial data belongs to $\mathcal{B}_1\cup \mathcal{B}_2$.  For any $a\in \Bbb{R}$, we can construct a function $u_0\in \mathcal{B}_2$ such that $J(u_0)=a$ (see Corollary  \ref{cor: arbitrary}).    In particular, we also derive the upper bound and lower bound for the blow up time.

The rest of this paper is organized as follows. In Section 2, we introduce some notations, definitions and conclusions that will be used in the sequel, including the result of local well-posedness. We establish a concavity argument for problem  \eqref{P1} in  Section 3.1, then give the criteria  of finite time blow-up in Section 3.2. Finally, combining with the  interpolation inequality of Gagliardo–Nirenberg, we derive the lower bound for the blow up time in Section 3.3.

\section{Preliminaries}
For convenience, we denote $\|\cdot\|_q=\|\cdot\|_{L^q(\Omega)}$, $\|\cdot\|_{q,\Gamma_1}=\|\cdot\|_{L^q(\Gamma_1)}$ for $1\leq q\leq \infty$, and the Hilbert space
\[ H_{\Gamma_0}^1(\Omega)=\left\{w\in H^1(\Omega)\ :\ w\big|_{\Gamma_0}=0  \right\},\]
where $w\big|_{\Gamma_0}$ stands for the restriction of the trace of $w$ on $\partial \Omega$ to $\Gamma_0$.
We also denote $(\cdot,\cdot)$ and $(\cdot,\cdot)_{\Gamma_1}$ as the inner products on the Hilbert spaces $L^2(\Omega)$ and $L^2(\Gamma_1)$ respectively.
 The trace theorem implies the existence of the continuous trace mapping
\[ H_{\Gamma_0}^1(\Omega) \hookrightarrow L^2(\partial \Omega) .\]
Since $\sigma\left(\Gamma_0\right)>0$, a Poincar\`{e}-type inequality  holds, see \cite[Theorem 6.7-5]{Ciarlet-NFA}, and consequently $\|\nabla w\|_2$ is equivalent to the norm
\[ \|w\|_{H_{\Gamma_0}^1} =\left(\|w\|_2^2+\|\nabla w\|_2^2  \right)^\frac{1}{2} \]
in the space $H_{\Gamma_0}^1(\Omega)$. Based on the above arguments, we can define the following positive optimal constants
\begin{equation}\label{eq: optimal constants}
S_1=\sup_{\substack{w\in H_{\Gamma_0}^1(\Omega)\\ w\neq 0}} \frac{\|w\|_{2,\Gamma_1}^2}{\|\nabla w\|_2^2},\quad S_2=\sup_{\substack{w\in H_{\Gamma_0}^1(\Omega)\\ w\neq 0}} \frac{\|w\|_{2}^2}{\|\nabla w\|_2^2}.
\end{equation}

We now introduce the definition of weak solution.
\begin{definition}\label{def: weak solution}
Assume that $u_0\in H_{\Gamma_0}^1(\Omega)$,
\[2\leq p\leq 1+\frac{2^*}{2},\]
where $2^*$ is the critical exponent of Sobolev embedding $H^1(\Omega)\hookrightarrow L^q(\Omega)$, i.e. $2^*=\frac{2n}{n-2}$ if $n\geq 3$; $2^*=\infty$ if  $n=1,2$. The function $u$ is said to be a weak solution of problem \eqref{P1} in $[0, T]\times \Omega$ if
\begin{itemize}
  \item[(a)] $u\in L^\infty (0,T;H_{\Gamma_0}^1(\Omega)$; $u_t\in L^2((0,T)\times \Omega)$;
  \item[(b)] the spatial trace of $u$ on $(0,T)\times \partial \Omega$ (which exists by the trace theorem) has a distributional time derivative $u_t$ on $(0,T)\times \partial \Omega$, belong to $L^2((0,T)\times \partial \Omega)$;
  \item[(c)] for all $\phi\in H_{\Gamma_0}^1(\Omega)$
  and for almost all $t\in [0,T]$ the distributional identity
\begin{equation}\label{eq: weak solution}
  (u_t(t),\phi) +( \nabla u(t), \nabla \phi) +(u_t(t),\phi)_{\Gamma_1}=\int_{\Omega} |u(t)|^{p-2}u(t)\phi
\end{equation}
  holds true;
  \item[(d)] $u(0)=u_0$.
\end{itemize}
\end{definition}
\begin{remark}
By (a) and the continuous embedding
\[ H^1(0,T; L^2(\Omega))\hookrightarrow C([0,T];L^2(\Omega)),\]
we have $u\in C([0,T];L^2(\Omega))$, therefore $u(0)$ makes sense in (d).
\end{remark}

The  following result of local well-posedness for problem \eqref{P1} can be obtained by   the monotonicity method of J.L. Lions and a contraction argument.
\begin{theorem}[{\cite[Theorem 2]{VitillaroDCDS13}}]\label{Thm: local well-posedness}
Let
\begin{equation}\label{eq: assumption for p}
2\leq p\leq 1+\frac{2^*}{2}.
\end{equation}
Then, for any $u_0\in H_{\Gamma_0}^1(\Omega)$, problem \eqref{P1} has a unique weak maximal solution $u$ in $[0,T_{max}) \times\Omega$, where $T_{max}$ is the maximal existence time for the weak solution. Moreover $u\in C([0, T_{max}); H_{\Gamma_0}^1(\Omega))$,
$$u_t\in L^2 ((0,T)\times \Gamma_1)\cap L^2((0,T)\times \Omega)\quad \textrm{for any }T\in (0,T_{max}),$$
the energy identity
\begin{equation}\label{eq: energy identity}
\frac{1}{2} \|\nabla u\|_2^2 \Big|_s^t +\int_s^t \left( \|u_t(\tau)\|_2^2+ \|u_t(\tau)\|_{2,\Gamma_1}^2\right) \d \tau=\int_s^t \int_\Omega |u|^{p-2}uu_t \d\tau\d x
\end{equation}
holds for $0\leq s\leq t< T_{max}$ and the following alternative holds:
\begin{itemize}
  \item[(i)] either $T_{max}=\infty$;
  \item[(ii)] or $T_{max}<\infty$ and
  \begin{equation*}
    \lim_{t\to T_{max}^-} \|u(t)\|_{H_{\Gamma_0}^1}=+\infty.
  \end{equation*}
\end{itemize}
\end{theorem}
\begin{remark}
In fact, if  $p=2$, the weak solution $u$ is global, see \cite[Appendix B]{VitillaroDCDS13}.
\end{remark}

When $2<p \leq   1+\dfrac{2^*}{2}$, we  introduce the following functionals
\begin{equation}\label{eq: energy functionals}
J(w)=\frac{1}{2}\|\nabla w\|_2^2 -\frac{1}{p}\|w\|_p^p,\quad K(w)=\|\nabla w\|_2^2 -\|w\|_p^p,\quad w\in H_{\Gamma_0}^1(\Omega).
\end{equation}
According to Definition \ref{def: weak solution} and Theorem \ref{Thm: local well-posedness}, we obtain the following lemma.
\begin{lemma}\label{lem: diff identities}
Assume that $u_0\in H_{\Gamma_0}^1(\Omega)$,
\[2< p\leq 1+\frac{2^*}{2}\]
and $u$ is the unique maximal weak solution of problem \eqref{P1} in $[0,T_{max})\times \Omega$. Then
\begin{itemize}
  \item[(i)]
  \[\frac{\d}{\d t} \|u(t)\|_p^p =p\int_\Omega |u(t)|^{p-2} u(t) u_t(t) \d x\quad \textrm{for a.e. }t\in (0,T_{max});\]
  \item[(ii)]
 \begin{equation}\label{eq: energy identity -diff}
  \frac{\d}{\d t}J(u(t))=-\left( \|u_t(t)\|^2_{2,\Gamma_1}+\|u_t(t)\|_2^2\right)\leq 0 \quad \textrm{for a.e. }t\in (0,T_{max});
 \end{equation}
  \item[(iii)]
  \[ \frac{\d}{\d t} \rho(t) =(u(t),u_t(t))+ (u(t),u_t(t))_{\Gamma_1}=-K(u(t))\quad \textrm{for a.e. }t\in (0,T_{max}),\]
  where
  \[\rho(t)=  \frac{1}{2} \|u(t)\|_2^2 +\frac{1}{2}  \|u(t)\|_{2,\Gamma_1}^2.\]
\end{itemize}
\end{lemma}
\begin{proof}
For the proof of statement (i), we refer to \cite[Lemma 3]{VitillaroDCDS13}.

By  statement (i) and (\ref{eq: energy identity}), we have
\begin{equation}\label{eq: energy identity integral}
J(u(t))-J(u_0)=-\int_0^t\left( \|u_t(\tau)\|_2^2+ \|u_t(\tau)\|_{2,\Gamma_1}^2\right) \d \tau
\end{equation}
for any $t\in [0,T_{max})$. In view of the regularity of the weak solution, the function $t\mapsto J(u(t))$ is absolutely continuous on $[0, T_{max})$  and therefore (\ref{eq: energy identity -diff}) holds.

Since $u\in C([0,T_{max}); H_{\Gamma_0}^1(\Omega))$ and
$$u_t\in L^2 ((0,T)\times \Gamma_1)\cap L^2((0,T)\times \Omega)\quad \textrm{for any }T\in (0,T_{max}),$$
the function $\rho$ is of absolutely continuous  on $(0, T_{max})$ and
$$ \frac{\d}{\d t} \rho(t)=(u(t),u_t(t))+ (u(t),u_t(t))_{\Gamma_1}\quad \textrm{for a.e. }t\in (0,T_{max}).$$
Taking $\phi=u(t)$ in (\ref{eq: weak solution}), we get
\[ \frac{\d}{\d t} \rho(t)=(u,u_t)+ (u,u_t)_{\Gamma_1}=-K(u(t))\quad \textrm{for a.e. }t\in (0,T_{max}).\]

\end{proof}

Finally, we introduce the  potential well depth by
\begin{equation*}  
 d=\inf_{w\in N} J(w)=\inf_{w\in H_{\Gamma_0}^1 (\Omega)\setminus \{0\}} \sup_{\lambda >0} J(\lambda w),
\end{equation*}
where $N$ is the Nehari manifold
\[
N=\{ w\in H_{\Gamma_0}^1(\Omega)\ |\ K(w)=0\}\setminus\{0\}.
\]
It is easy to verify (see \cite[Lemma 2]{VitillaroDCDS13})
$$d=\left( \frac{1}{2}-\frac{1}{p}\right)B_1^{-\frac{2p}{p-2}}>0,$$
where
$$B_1= \sup_{\substack{w\in H_{\Gamma_0}^1(\Omega)\\ w\neq 0}} \frac{\|w\|_p}{\|\nabla w\|_{2}}.$$

\begin{lemma}\label{lem: potential well depth d}
When $2<p \leq   1+\dfrac{2^*}{2}$, it holds that
\[\| w\|_p^p>\frac{2p}{p-2}d\]
for any
\[w\in N_-=\{  w\in H_{\Gamma_0}^1(\Omega)\ |\ K(w)<0\}.\]
\end{lemma}
\begin{proof}
Since
\[\|\nabla w\|_2^2 -\|w\|_p^p=K(w)<0,\]
we have $w\neq 0$ and therefore $K(\lambda^*w)=0$ with
\[ \lambda^*=\left(\frac{\|\nabla w\|_2^2}{\|w\|_p^p}  \right)^{\frac{1}{p-2}}\in (0,1),\]
i.e. $\lambda^*w\in N$. In view of   the definition of the potential well depth $d$, it holds
\begin{align*}
  d=\inf_{w\in N}J(w)\leq  & J(\lambda^*w)\\
  =&\frac{p-2}{2p}\left\|\lambda^* w\right\|_p^p +\frac{1}{2}K(\lambda^*w)\\
  =&\frac{p-2}{2p}(\lambda^*)^p\left\|  w\right\|_p^p\\
  <&\frac{p-2}{2p}\| w\|_p^p.
\end{align*}
So we consequently obtain
\[\| w\|_p^p>\frac{2p}{p-2}d.\]
\end{proof}

The set $N_-$ plays an important role in the study of finite time blow-up. In fact, we have the following proposition for the necessary condition of finite time blow-up.

\begin{proposition}\label{prop: necessary condition}
Assume that
\[2< p\leq 1+\frac{2^*}{2}\]
and the weak solution $u$ of  problem (\ref{P1}) blows up in finite time.  Then, there exists $t^*\in [0, T_{max})$ such that
\[ u(t^*)\in  N_-.\]
\end{proposition}
\begin{proof}
Suppose, by the contrary, that
\[ K(u(t))\geq 0\quad \textrm{for all }t\in [0, T_{max}).\]
Then, by Lemma \ref{lem: diff identities}(ii), we have
\begin{align*}
  J(u_0)\geq J(u(t)) & =\frac{1}{2}\|\nabla u(t)\|_2^2 -\frac{1}{p}\|u(t)\|_p^p \\
  & =\frac{p-2}{2p}\|\nabla u\|_2^2+\frac{1}{p}K(u(t))\\
  &\geq \frac{p-2}{2p}\|\nabla u(t)\|_2^2
\end{align*}
for any $t\in [0,T_{max})$, which contradicts Theorem \ref{Thm: local well-posedness}(ii).
\end{proof}

\section{The finite time blow-up results}

\subsection{A concavity  argument}

We will apply Levine's concavity method to obtain the finite time blow-up results.
\begin{lemma}[\cite{Levine1973}]\label{lem: Levine concavity lem}
Assume that a positive function $F$ on $[0,T]$ satisfies the following conditions:
\begin{itemize}
\item[(i)]$F$ is differentiable on $[0,T]$, and $F'$ is absolutely continuous on $[0,T]$ with $F'(0)>0$;
  \item[(ii)] there exists a positive constant $\alpha>0$ such that
  $$F(t)F''(t)-(1+\alpha) \left(F'(t)\right)^2\geq 0\quad \textrm{for a.e }t\in [0,T].$$
\end{itemize}
 Then
 $$ T\leq \frac{F(0)}{\alpha F'(0)}.$$
\end{lemma}
\begin{remark}
The condition (ii) in Lemma \ref{lem: Levine concavity lem} implies that the function
$$G(t)=\left( F(t) \right)^{-\alpha}$$
is concave on $[0,T]$.
\end{remark}

Suppose that
\[2< p\leq 1+\frac{2^*}{2}\]
and $u$ is the unique maximal weak solution of problem \eqref{P1} with initial data $u_0\in H_{\Gamma_0}^1(\Omega)$. Denote
\[\rho(t)=  \frac{1}{2} \|u(t)\|_2^2 +\frac{1}{2}  \|u(t)\|_{2,\Gamma_1}^2,\quad t\in [0,T_{max}).\]
According to Lemma \ref{lem: diff identities} (iii), it holds that
\[\frac{\d}{\d t}\rho(t)=(u,u_t)+ (u,u_t)_{\Gamma_1}=-K(u(t))\quad \textrm{for a.e. }t\in (0,T_{max}).\]
Choose  an arbitrary $T$ such that
\[ 0<T<T_{max}\]
and define the auxiliary function
\begin{equation}\label{eq: F}
 F(t)=\int_{0}^t \rho(\tau)\d \tau +(T-t)\rho(0)+\frac{1}{2}\beta (t+\sigma)^2,\quad t\in [0,T],
\end{equation}
where $\beta$ and $\sigma$ are the positive parameters to be determined later. In view of \eqref{eq: energy functionals},  \eqref{eq: energy identity integral} and Lemma \ref{lem: diff identities} (iii), we have
\begin{align}\label{eq: F'}
F'(t)=&\rho(t)-\rho(0)+\beta (t+\sigma)\nonumber\\
      =&\int_{0}^ t \frac{\d}{\d t}\rho(\tau) \d\tau +\beta (t+\sigma)\nonumber\\
      =&\int_{0}^t (u,u_t) \d\tau+ \int_{0}^t  (u,u_t)_{\Gamma_1} \d\tau+\beta (t+\sigma)
\end{align}
for any $t\in [t_0, T]$ and
\begin{align}\label{eq: F''}
  F''(t) =&\frac{\d}{\d t}\rho(t)+\beta  \nonumber\\
          =&-K(u(t))+\beta \nonumber \\
          =&-\left(pJ(u(t))-\frac{p-2}{2}\|\nabla u(t)\|_2^2  \right)+\beta \nonumber \\
          =&\frac{p-2}{2}\|\nabla u(t)\|_2^2-p\left[J(u(0))-\int_{0}^t \left(\|u_t(\tau)\|_2^2 +  \|u_t(\tau)\|_{2,\Gamma_1}^2  \right) \d\tau \right]+\beta\nonumber\\
          =&\frac{p-2}{2}\|\nabla u(t)\|_2^2+p\int_{0}^t \|u_t(\tau)\|_2^2 \d\tau +  p\int_{0}^t \|u_t(\tau)\|_{2,\Gamma_1}^2 \d\tau -pJ(u_0)+\beta
\end{align}
for a.e. $t\in (0, T]$. Note that the auxiliary function $F$ is positive on $[0,T]$ and $F'(0)=\beta\sigma>0$.

Now  we derive an estimation for $FF''-\lambda\left(F'\right)^2$, where $\lambda>0$ is  a constant to be determined later.
Using the Cauchy-Schwartz inequality and Young's inequality, we obtain
\begin{align*}
  \xi(t)  =& \left[\int_{0}^t \|u(\tau)\|_2^2\d\tau + \int_{0}^t \|u(\tau)\|_{2,\Gamma_1}^2\d\tau+\beta (t+\sigma)^2 \right]\\
              &\cdot \left[ \int_{0}^t \|u_t(\tau)\|_2^2\d\tau + \int_{0}^t \|u_t(\tau)\|_{2,\Gamma_1}^2 \d\tau +\beta\right]\\
              &-\left[\int_{0}^t (u,u_t)\d\tau + \int_{0}^t (u,u_t)_{\Gamma_1}\d\tau+\beta (t+\sigma) \right]^2\\
           = & \left[\int_{0}^t \|u(\tau)\|_2^2 \d\tau \cdot \int_{0}^t \|u_t(\tau)\|_2^2-\left(\int_{0}^t (u,u_t)\d\tau  \right)^2  \right]\\
              &+ \left[ \int_{0}^t \|u(\tau)\|_{2,\Gamma_1}^2 \d\tau \cdot \int_{0}^t \|u_t(\tau)\|_{2,\Gamma_1}^2\d\tau-\left(\int_{0}^t (u,u_t)_{\Gamma_1}\d\tau  \right)^2   \right]\\
              &+  \left[ \int_{0}^t \|u(\tau)\|_2^2 \d\tau \cdot \int_{0}^t \|u_t(\tau)\|_{2,\Gamma_1}^2 \d\tau +\int_{0}^t \|u(\tau)\|_{2,\Gamma_1}^2 \d\tau\cdot\int_{0}^t \|u_t(\tau)\|_2^2\d\tau   \right.\\
              &\quad\left. -2\int_{0}^t (u,u_t)\d\tau \cdot \int_{0}^t (u,u_t)_{\Gamma_1}\d\tau\right]\\
              &+\left[\beta(t+\sigma)^2\int_{0}^t \|u_t(\tau)\|_2^2 \d\tau +\beta\int_{0}^t \|u(\tau)\|_{2}^2\d\tau -2\beta(t+\sigma)\int_{0}^t (u,u_t)\d\tau \right]\\
              &+ \left[\beta(t+\sigma)^2\int_{0}^t \|u_t(\tau)\|_{2,\Gamma_1}^2\d\tau +\beta\int_{0}^t \|u(\tau)\|_{2,\Gamma_1}^2\d\tau -2\beta(t+\sigma)\int_{0}^t (u,u_t)_{\Gamma_1}\d\tau \right]\\
           \geq &0
\end{align*}
for any $t\in [0,T]$.Therefore, in view of  (\ref{eq: F})-(\ref{eq: F''}), it holds that
\begin{align*}
   & FF''-\lambda \left(F'\right)^2 \\
\overset{\eqref{eq: F'}}{=} & FF''-\lambda \left[\int_{0}^t (u,u_t)\d\tau +\int_{0}^t (u,u_t)_{\Gamma_1}\d\tau+\beta (t+\sigma) \right]^2\\
= & FF''+\lambda \left[\xi(t)-\left(\int_{0}^t \|u(\tau)\|_2^2 + \int_{0}^t \|u(\tau)\|_{2,\Gamma_1}^2\d\tau +\beta (t+\sigma)^2  \right) \right. \\
  &  \quad \quad \quad \quad \quad\cdot\left.\left(\int_{0}^t \|u_t(\tau)\|_2^2\d\tau + \int_{0}^t \|u_t(\tau)\|_{2,\Gamma_1}^2 \d\tau +\beta  \right)\right]\\
\overset{\eqref{eq: F}}{=} & FF''+\lambda \left[\xi(t)-2\left(F(t)-(T-t)\rho(0)\right)\cdot \left(\int_{0}^t \|u_t(\tau)\|_2^2\d\tau +\int_{0}^t \|u_t(\tau)\|_{2,\Gamma_1}^2 \d\tau +\beta  \right)  \right]\\
\geq &FF''-2\lambda F(t)\cdot  \left(\int_{0}^t \|u_t(\tau)\|_2^2\d\tau +\int_{0}^t \|u_t(\tau)\|_{2,\Gamma_1}^2 \d\tau +\beta  \right)  \\
\overset{\eqref{eq: F''}}{=}&F(t)\cdot \left[ \left(\frac{p-2}{2}\|\nabla u(t)\|_2^2 +p\int_{0}^t \|u_t(\tau)\|_2^2\d\tau + p\int_{0}^t \|u_t(\tau)\|_{2,\Gamma_1} \d\tau -pJ(u_0)+\beta\right)\right.\\
&\quad\quad \quad \quad \quad  \left. -2\lambda\left( \int_{0}^t \|u_t(\tau)\|_2^2 \d\tau +  \int_{0}^t \|u_t(\tau)\|_{2,\Gamma_1}^2\d\tau + \beta\right)\right]\\
=& F(t)\cdot \left[\frac{p-2}{2}\|\nabla u(t)\|_2^2 +(p-2\lambda) \int_{0}^t \|u_t(\tau)\|_2^2\d\tau +(p-2\lambda)\int_{0}^t \|u_t(\tau)\|_{2,\Gamma_1}^2 \d\tau -pJ(u_0)+(1-2\lambda)\beta\right].
\end{align*}
Taking $\lambda =\dfrac{p}{2}$, we finally obtain the following estimation:
\begin{equation}\label{eq: concavity ineq}
FF''-\frac{p}{2}\left(F'\right)^2\geq F\cdot \left[\frac{p-2}{2}\|\nabla u(t)\|_2^2 -pJ(u_0)-(p-1)\beta\right] \quad \textrm{for a.e. }t\in [0,T] .
\end{equation}

\begin{lemma}\label{lem: concavity method}
Assume that  $\rho(0)>0$ and there exists  some positive parameter $\beta$  such that
\[\frac{p-2}{2}\|\nabla u(t)\|_2^2 -pJ(u_0)-(p-1)\beta\geq 0\]
holds for any $t\in (0,T]$. Then
\[0<T\leq   \frac{8\rho(0)}{(p-2)^2 \beta}.\]
\end{lemma}
\begin{proof}
Note that $\dfrac{p}{2}>1$. In view of  Lemma \ref{lem: Levine concavity lem} and \eqref{eq: concavity ineq}, we have
\begin{equation}\label{eq: estimate T -1}
T\leq \frac{F(0)}{\frac{p-2}{2}\cdot F'(0)}=\frac{T\rho(0)+\frac{1}{2}\beta\sigma^2}{\frac{p-2}{2}\cdot \beta\sigma}=\frac{2\rho(0)}{(p-2)\beta\sigma} T+\frac{\sigma}{p-2}.
\end{equation}
To guarantee
\[  \frac{2\rho(0)}{(p-2)\beta\sigma}<1,\]
we restrict the range of $\sigma$ to be $\left( \frac{2\rho(0)}{(p-2)\beta},+\infty\right)$. Therefore, by (\ref{eq: estimate T -1}), we obtain
\begin{equation}\label{eq: estimate T -2}
T\leq \left(1- \frac{2\rho(0)}{(p-2)\beta\sigma} \right)^{-1}\frac{\sigma}{p-2}.
\end{equation}
A direct calculation shows that the right hand side of (\ref{eq: estimate T -2}) takes its minimum at
\[ \sigma=\sigma_{\beta}=\frac{4\rho(0)}{(p-2)\beta} \in \left( \frac{2\rho(0)}{(p-2)\beta},+\infty\right).\]
Since $T$ is independent of the parameter $\sigma$, we finally obtain
\[ T\leq \left(1- \frac{2\rho(0)}{(p-2)\beta\sigma_\beta} \right)^{-1}\frac{\sigma_\beta}{p-2} =\frac{8\rho(0)}{(p-2)^2 \beta}.\]
\end{proof}

\subsection{The finite time blow-up criteria}

Based on Lemma \ref{lem: concavity method}, we give the following finite time blow-up criteria.

\begin{theorem}\label{thm: blow-up}
Assume that
\[2< p\leq 1+\frac{2^*}{2}\]
and the initial data $u_0$ belongs to one of the following sets:
\begin{align*}
  \mathcal{B}_1 & =\left\{ w\in H_{\Gamma_0}^1(\Omega)\ :\  J(w)< d\  \textrm{and}\  K(w)<0\right\}, \\
  \mathcal{B}_2 & =\left\{ w\in H_{\Gamma_0}^1(\Omega)\ :\  J(w)<\frac{p-2}{2p( S_1+S_2)}\left(\|w\|_2^2 +  \|w\|_{2,\Gamma_1}^2\right)\right\}.
\end{align*}
Then the weak solution $u$ of problem (\ref{P1}) blows up in finite time. Moreover, if $u_0\in \mathcal{B}_1$,  we have
\[T_{max}\leq \frac{4(p-1)}{p(p-2)^2} \cdot  \frac{\|u_0\|_2^2+  \|u_0\|_{2,\Gamma_1}^2}{d-J(u_0) };\]
if $u_0\in \mathcal{B}_2$, we have
\[T_{max}\leq \frac{4(p-1)}{p(p-2)^2}\cdot \frac{\|u_0\|_2^2 + \|u_0\|_{2,\Gamma_1}^2}{\frac{p-2}{2p( S_1+S_2)}\left(\|u_0\|_2^2 + \|u_0\|_{2,\Gamma_1}^2 \right)-J(u_0)}.  \]

\end{theorem}
\begin{proof}

\textbf{Part 1. The case of $u_0\in  \mathcal{B}_1$}.

First, we show that $u(t)\in \mathcal{B}_1$ holds for any $t\in [0, T_{max})$ provided $u_0\in \mathcal{B}_1$.

By the energy identity \eqref{eq: energy identity integral}, we have
\begin{equation}\label{eq: J(u(t))<d -1}
J(u(t))=J(u_0)-\int_{0}^t \left( \|u_t(\tau)\|_2^2+ \|u_t(\tau)\|_{2,\Gamma_1}^2\right) \d \tau\leq J(u_0)< d
\end{equation}
for any $t\in [0, T_{max})$. We only need to prove that $K(u(t))<0$ for any $t\in [0, T_{max})$.
Since $u\in C([0, T_{max}); H_{\Gamma_0}^1(\Omega))$, the mapping $t\mapsto K(u(t))$ is continuous on $[0, T_{max})$. Suppose, on the contrary, that there exists $t_1\in (0, T_{max})$ such that
\[K(u(t))<0\quad \textrm{for any }t\in [0,t_1);\quad K(u(t_1))=0.\]
Considering  Lemma \ref{lem: potential well depth d} and the continuity of the mapping $t\mapsto \|\nabla u(t)\|_2$, we have
\[ \|\nabla u(t_1)\|_2^2=\lim_{t\to t_1^{-}}\|\nabla u(t)\|_2^2\geq \frac{2p}{p-2}d>0,\]
so $u(t_1)\in N$.  While, the definition of the potential well depth $d$ shows that
\[ d=\inf_{w\in N}J(w)\leq J(u(t_1)),\]
which contradicts \eqref{eq: J(u(t))<d -1}.

According to the above argument,  $u(t)\in \mathcal{B}_1$ for any $t\in [0, T_{max})$ provided $u_0\in \mathcal{B}_1$.  It is obvious that $\rho(0)>0$. In view of Lemma \ref{lem: potential well depth d}, we have
\[\frac{p-2}{2}\|\nabla u(t)\|_2^2 -pJ(u_0)-(p-1)\beta> pd-pJ(u_0)-(p-1)\beta\geq 0\]
for any $t\in (0, T_{max})$ and any
\[\beta\in \left(0,\frac{p}{p-1}\left(d-J(u_0)\right)  \right]. \]
According to  Lemma \ref{lem: concavity method}, we obtain
\[ 0<T\leq \frac{8(p-1)\rho(0)}{p(p-2)^2\left(d-J(u_0)\right) }<\infty\]
for any $T\in (0, T_{max})$. So the maximal existence time $T_{max}$ of the weak solution $u$  satisfies that
\[ 0<T_{max}\leq \frac{4(p-1)}{p(p-2)^2} \cdot  \frac{\|u_0\|_2^2+  \|u_0\|_{2,\Gamma_1}^2}{d-J(u_0) }, \]
hence $u$  blows up in finite time.

\textbf{Part 2. The case of $u_0\in  \mathcal{B}_2$}.

By Lemma \ref{lem: diff identities} (iii) and (\ref{eq: optimal constants}), we have
\begin{align}\label{eq: estimate rho'(t) -1}
  \frac{\d}{\d t}\rho(t)  =-K(u(t)) & =\frac{p-2}{2}\|\nabla u(t)\|_2^2 -pJ(u(t)) \nonumber \\
  &\geq \frac{p-2}{2} \cdot \frac{2}{  S_1+S_2}\rho(t) -pJ(u(t))\nonumber\\
  &=\frac{p}{A} \Big(\rho(t)-AJ(u(t))\Big)
\end{align}
for a.e.  $t\in (0,T_{max})$, where
\[ A= \frac{p( S_1+S_2)}{p-2}>0.\]
Let $H(t)=\rho(t)-AJ(u(t))$. In view of Lemma \ref{lem: diff identities} (ii) and (\ref{eq: estimate rho'(t) -1}), we obtain
\begin{equation*}
\frac{\d}{\d t} H(t)  =\frac{\d}{\d t}\rho(t)-A\frac{\d}{\d t} J(u(t))
   \geq \frac{\d}{\d t}\rho(t) \geq \frac{p}{A} H(t)
\end{equation*}
for a.e. $t\in (0, T_{max})$. Using  Gronwall's inequality, we have $H(t)\geq e^{\frac{p}{A}t} H(0)$. The assumption $u_0\in \mathcal{B}_2$ implies that
\[ H(0)=\rho(0)-AJ(u_0)=\frac{1}{2}\left( \|u_0\|_2^2 +  \|u_0\|_{2,\Gamma_1}^2 \right) - \frac{p(  S_1+S_2)}{p-2}J(u_0)>0,\]
so
\[ \frac{\d}{\d t}\rho(t) \geq \frac{p}{A} H(t)\geq \frac{p}{A}e^{\frac{p}{A}t} H(0)>0\]
for a.e.  $t\in (0, T_{max})$, which means that  $\rho(t)$ is nondecreasing on $[0, T_{max})$.

In view of   (\ref{eq: optimal constants}) and   the monotonicity of $\rho(t)$, it follows that
\begin{align*}
  &\frac{p-2}{2}\|\nabla u(t)\|_2^2 -pJ(u_0)-(p-1)\beta   \\
\geq  &  \frac{p-2}{2} \cdot \frac{2}{  S_1+S_2}\rho(t)-pJ(u_0)-(p-1)\beta \\
\geq & \frac{p-2}{2} \cdot \frac{2}{  S_1+S_2}\rho(0)-pJ(u_0)-(p-1)\beta \\
= & \frac{p}{A}H(0)-(p-1)\beta \\
\geq &0
\end{align*}
for any  $t\in (0, T_{max})$ and any
\[\beta\in \left(0, \frac{pH(0)}{A(p-1)}\right].\]
According to  Lemma \ref{lem: concavity method} , we obtain
\[0<T\leq \frac{8A(p-1)\rho(0)}{p(p-2)^2H(0)}\]
for any $T\in (0, T_{max})$. So the maximal existence time $T_{max}$ of the weak solution $u$  satisfies that
\[ 0<T_{max}\leq \frac{4(p-1)}{p(p-2)^2}\cdot \frac{\|u_0\|_2^2 + \|u_0\|_{2,\Gamma_1}^2}{\frac{p-2}{2p( S_1+S_2)}\left(\|u_0\|_2^2 + \|u_0\|_{2,\Gamma_1}^2 \right)-J(u_0)} ,\]
hence  $u$  blows up in finite time.
\end{proof}

It is obvious that both $\mathcal{B}_1$ and $\mathcal{B}_2$ are none-empty  sets. Moreover, the following corollary implies that,  for any $a\in \Bbb{R}$, there exists $u_0\in H_{\Gamma_0}^1(\Omega)$ with  initial energy $J(u_0)=a$  which leads to finite time blow-up solution.

\begin{corollary}\label{cor: arbitrary}
For any $a\in \Bbb{R}$, denote the energy level set by
\[J^a:=\{w\in H_{\Gamma_0}^1(\Omega)\,|\, J(w)=a\}.\]
Then $J^a\cap \mathcal {B}_2\neq \emptyset$.
\end{corollary}
\begin{proof}
Assume that $\Omega_1$ and $\Omega_2$ are two disjoint open subdomains of $\Omega$, and
\[ {\rm dist}\left(\overline{\Omega_1}, \partial \Omega\right)>0,\quad  {\rm dist}\left(\overline{\Omega_2}, \partial \Omega\right)>0,\quad  {\rm dist}\left(\overline{\Omega_1}, \overline{\Omega_2}\right)>0. \]
According to the proof of Theorem 3.7 in \cite{WillemMinimaxBook}, there exists a sequence $\{v_k\}\subset H_0^1(\Omega_1)$ such that
 \begin{equation}\label{eq: high energy -1}
 \frac{1}{2} \int_{\Omega_1} |\nabla v_k(x)|^2 \d x-\frac{1}{p} \int_{\Omega_1} |v_k(x)|^p \d x\to +\infty \quad \textrm{as }k\to \infty.
 \end{equation}
On the other hand, choosing an arbitrary nonzero function $w\in C_0^\infty(\Omega)$ with the support ${\rm supp}\, w\subset \Omega_2$,  then
\begin{equation}\label{eq: high energy -2}
a- \left(  \frac{1}{2} \int_{\Omega_2} |\nabla (rw(x))|^2 \d x-\frac{1}{p} \int_{\Omega_2} |rw(x)|^p \d x \right)\to +\infty \quad\textrm{as }r\to +\infty,
\end{equation}
and  there exists $r_0>0$ such that
\begin{equation}\label{eq: high energy -3}
\frac{p-2}{2p(  S_1+S_2)} \int_{\Omega_2} |rw(x)|^2 \d x= r^2\cdot \frac{p-2}{2p( S_1+S_2)} \int_{\Omega_2} |w(x)|^2 \d x>a
\end{equation}
for any $r>r_0$. By (\ref{eq: high energy -1}) and (\ref{eq: high energy -2}), there exist $k_0\in\Bbb{N}_+$ and $r_1>r_0$ such that
\begin{equation}\label{eq: high energy -4}
\frac{1}{2} \int_{\Omega_1} |\nabla v_{k_0}(x)|^2 \d x-\frac{1}{p} \int_{\Omega_1} |v_{k_0}(x)|^p \d x
=  a-\left(  \frac{1}{2} \int_{\Omega_2} |\nabla (r_1w(x))|^2 \d x-\frac{1}{p} \int_{\Omega_2} |r_1w(x)|^p \d x \right)
\end{equation}
Let $u_0=\tilde{v}+r_1w$, where
\[ \tilde{v}(x)=
\begin{cases}
0,&x\in \Omega\setminus \Omega_1\\
v_{k_0}(x),&x\in \Omega_1.
\end{cases}  \]
It is easy to verify that $u_0\in H_{\Gamma_0}^1(\Omega)$ and $u_0(x)=0$ on $\Omega\setminus\left(\Omega_1\cup \Omega_2\right)$. In view of (\ref{eq: high energy -3}) and (\ref{eq: high energy -4}), we finally obtain
\begin{align*}
  J(u_0) &= \frac{1}{2}\left(\int_{\Omega_1}+\int_{\Omega_2}\right)|\nabla u_0(x)|^2\d x- \frac{1}{p}\left(\int_{\Omega_1}+\int_{\Omega_2}\right)|u_0(x)|^p \d x \\
 & = \left( \frac{1}{2} \int_{\Omega_1} |\nabla v_{k_0}(x)|^2 \d x-\frac{1}{p} \int_{\Omega_1} |v_{k_0}(x)|^p \d x\right)+ \left(  \frac{1}{2} \int_{\Omega_2} |\nabla (r_1w(x))|^2 \d x-\frac{1}{p} \int_{\Omega_2} |r_1w(x)|^p \d x \right)\\
 &\overset{\eqref{eq: high energy -4}}{=}a\\
 &\overset{\eqref{eq: high energy -3}}{<} \frac{p-2}{2p(  S_1+S_2)} \int_{\Omega_2} |r_1w(x)|^2 \d x\\
 &= \frac{p-2}{2p(  S_1+S_2)} \int_{\Omega_2} |u_0(x)|^2\d x\\
 &\leq \frac{p-2}{2p(  S_1+S_2)}  \left(\|u_0\|_2^2+ \|u_0\|_{2,\Gamma_1}^2\right),
\end{align*}
i.e., $u_0\in J^a\cap \mathcal{B}$. This completes the proof.
\end{proof}

\subsection{Lower bound of the blow up time}

In view of the proof of Theorem \ref{thm: blow-up},  the sets $\mathcal{B}_1$ and $\mathcal{B}_2$ are  invariant under the semi-flow associated with   problem  (\ref{P1}), that is to say, $u(t)\in \mathcal{B}_1$ for any $t\in (0,T_{max})$ provided the initial data
$u_0\in \mathcal{B}_1$, and  $u(t)\in \mathcal{B}_2$ for any $t\in (0,T_{max})$ provided the initial data
$u_0\in \mathcal{B}_2$. On the other hand,  by (\ref{eq: optimal constants}), we have
\begin{align*}
  K(w) & = \|\nabla w\|_2^2 -\|w\|_p^p \\
   & =pJ(w) -\frac{p-2}{2}\|\nabla w\|_2^2\\
   & \leq p\left[ J(w)-\frac{p-2}{2p(   S_1+S_2)}\left(\|w\|_2^2 + \|w\|_{2,\Gamma_1}^2\right) \right]\\
   &<0
\end{align*}
for the case of $w\in \mathcal{B}_2$. Therefore, it holds that
\[ u(t) \in N_-=\{  w\in H_{\Gamma_0}^1(\Omega)\ |\ K(w)<0\}\]
for any $t\in [0, T_{max})$ provided the initial data $u_0\in \mathcal{B}_1 \cup \mathcal{B}_2$.  Based on the above arguments, it is natural to ask whether
or not the condition $u_0\in N_-$ is sufficient enough for finite time blow-up.  This is not an easy task and  we only refer the  reader to \cite{Dickstein2011} for similar research.

Now we derive  the lower bound of the blow-up time.

\begin{theorem}\label{thm: lower bound}
Assume  that
\[2< p\leq 1+\frac{2^*}{2},\quad  p<2+\frac{4}{n}, \]
the weak solution $u$ of  problem (\ref{P1}) blows up in finite time and $u(t)\in N_-$ for any $t\in [0, T_{max})$. Then
\[ T_{max}\geq \frac{\tilde{C}}{\left(\|u_0\|_2^2+\|u_0\|_{2,\Gamma_1}\right)^{\frac{2(p-2)}{4-n(p-2)}}},\]
where $\tilde{C}$ is a positive constant that will be determined in the proof.
\end{theorem}
\begin{proof}
Since $u(t)\in N_-$ for any $t\in [0,T_{max})$, considering Lemma \ref{lem: potential well depth d} and the Sobolev embedding theorem, we have $\|\nabla u(t)\|_2>0$, $\|u(t)\|_p>0$ and
$$\rho(t)= \frac{1}{2} \|u(t)\|_2^2 +\frac{1}{2}  \|u(t)\|_{2,\Gamma_1}^2\geq \frac{1}{2}\|u(t)\|_2^2> 0$$
for any $t\in [0, T_{max})$. Using the interpolation inequality of Gagliardo–Nirenberg, we have
\begin{equation}\label{eq: G-N ineq}
\|\nabla u(t)\|_2^2<\|u(t)\|_p^p\leq S_3 \|u(t)\|_2^{p(1-\sigma)} \|\nabla u(t)\|_2^{p\sigma}
\end{equation}
for any $t\in [0,T_{max})$, where $S_3$ is a positive constant and $\sigma=\dfrac{n(p-2)}{2p}$. Since $p<2+\frac{4}{n}$,  it is easy to verify that $2-p\sigma>0$, so
\[ \|\nabla u(t)\|_2\leq S_3^{\frac{1}{2-p\sigma}}\|u(t)\|_2^\frac{p-p\sigma}{2-p\sigma}\quad\textrm{for any }t\in [0,T_{max}).\]
 From Lemma \ref{lem: diff identities} (iii) and \eqref{eq: G-N ineq}, it follows that
\begin{align*}
  \frac{\d}{\d t}\rho(t) & = -K(u(t))\\
   &=\|u(t)\|_p^p-\|\nabla u(t)\|_2^2\\
   &< \|u(t)\|_p^p\\
   & \leq S_3 \|u(t)\|_2^{p(1-\sigma)} \|\nabla u(t)\|_2^{p\sigma}\\
   &\leq S_3 \|u(t)\|_2^{p(1-\sigma)}\cdot S_3^{\frac{p\sigma}{2-p\sigma}}\|u(t)\|_2^\frac{p\sigma(p-p\sigma)}{2-p\sigma}\\
   &=S_3^{\frac{2}{2-p\sigma}} \|u(t)\|_2^{2\cdot \frac{p-p\sigma}{2-p\sigma}}\\
   &\leq S_3^{\frac{2}{2-p\sigma}} \left[2\rho(t)\right]^{\frac{p-p\sigma}{2-p\sigma}}\\
   &={S_4} \left[ \rho(t)\right]^{\frac{p-p\sigma}{2-p\sigma}}
\end{align*}
for a.e. $t\in [0, T_{max})$, where
\[ {S_4}= S_3^{\frac{2}{2-p\sigma}}\cdot 2^{\frac{p-p\sigma}{2-p\sigma}}>0\]
Noting that $\dfrac{p-p\sigma}{2-p\sigma}>1$, integrating the  differential inequality
\[\frac{\d}{\d t}\rho(t) <{S_4} \left[ \rho(t)\right]^{\frac{p-p\sigma}{2-p\sigma}},\]
we  then obtain
\begin{equation}\label{eq: rho(t) est}
\rho^{-\frac{p-2}{2-p\sigma}}(t)-\rho^{-\frac{p-2}{2-p\sigma}}(0)> -{S_4}\cdot {\frac{p-2}{2-p\sigma}}  t
\end{equation}
for any $t\in (0,T_{max})$. Since the weak solution $u$ blows up in finite time, from Theorem \ref{Thm: local well-posedness} (ii), it follows that $\lim\limits_{t\to (T_{max})^-}\rho(t)=+\infty$.
Letting $  t\to (T_{max})^-$ in (\ref{eq: rho(t) est}), we then obtain
\[-\rho^{-\frac{p-2}{2-p\sigma}}(0) \geq  -{S_4}\cdot {\frac{p-2}{2-p\sigma}}  T_{max}.\]
Finally, we  have
\[ T_{max}\geq \frac{2-p\sigma}{{S_4}(p-2)} \rho^{-\frac{p-2}{2-p\sigma}}(0)=\frac{\tilde{C}}{\left(\|u_0\|_2^2+\|u_0\|_{2,\Gamma_1}\right)^{\frac{2(p-2)}{4-n(p-2)}}},\]
where
\[ \tilde{C}=\frac{2-p\sigma}{{S_4}(p-2)} \cdot 2^{\frac{p-2}{2-p\sigma}}>0.\]
\end{proof}

\section*{Acknowledgment}
The authors would like to thank the anonymous referee for some valuable comments and suggestions. This work is partially supported by the National Natural
Science Foundation of China (11871302), the  Natural Science
Foundation of Shandong Province of China (ZR2019BA029) and the China Postdoctoral Science Foundation
(2021M701964).



\begin{thebibliography}{10}

\bibitem{Ioan-EJDE01}
I.~Bejenaru, J.~I. Diaz, and I.~I. Vrabie.
\newblock An abstract approximate controllability result and applications to
  elliptic and parabolic systems with dynamic boundary conditions.
\newblock {\em Electron. J. Differ. Eq.}, 2001(50):1--19, 2001.

\bibitem{Cavalcanti-JDE07}
M.~M. Cavalcanti, V.~N. Domingos~Cavalcanti, and I.~Lasiecka.
\newblock Well-posedness and optimal decay rates for the wave equation with
  nonlinear boundary damping–source interaction.
\newblock {\em J. Differ. Equ.}, 236(2):407--459, 2007.

\bibitem{Chueshov-CPDE02}
I.~Chueshov, M.~Eller, and I.~Lasiecka.
\newblock On the attractor for a semilinear wave equation with critical
  exponent and nonlinear boundary dissipation.
\newblock {\em Commun Part. Diff. Eq.}, 27(9-10):1901--1951, 2002.

\bibitem{Ciarlet-NFA}
P.~G. Ciarlet.
\newblock {\em {Linear and Nonlinear Functional Analysis with Applications}}.
\newblock SIAM-Society for Industrial and Applied Mathematics, 2013.

\bibitem{Dickstein2011}
F.~Dickstein, N.~Mizoguchi, P.~Souplet, and F.~Weissler.
\newblock Transversality of stable and {Nehari} manifolds for a semilinear heat
  equation.
\newblock {\em Calc. Var. Partial Dif.}, 42(3):547--562, 2011.

\bibitem{VitillaroDCDS13}
A.~Fiscella and E.~Vitillaro.
\newblock Local {Hadamard} well-posedness and blow--up for reaction-diffusion
  equations with non-linear dynamical boundary conditions.
\newblock {\em Discrete Contin. Dyn. Syst.-Ser. A}, 33(11$\&$12):5015--5047,
  2013.

\bibitem{Hintermann-PRES89}
T.~Hintermann.
\newblock Evolution equations with dynamic boundary conditions.
\newblock {\em P. Roy. Soc. Edinb. A}, 113(1-2):43--60, 1989.

\bibitem{Levine1973}
H.~A. Levine.
\newblock Some nonexistence and instability theorems for solutions of formally
  parabolic equations of the form {$P u_t=-Au+F(u)$}.
\newblock {\em Arch. Ration. Mech. Anal.}, 51(5):371--386, 1973.

\bibitem{Levine-JDE74}
H.~A. Levine and L.~E. Payne.
\newblock Nonexistence theorems for the heat equation with nonlinear boundary
  conditions and for the porous medium equation backward in time.
\newblock {\em J. Differ. Equ.}, 16(2):319--334, 1974.

\bibitem{Levine-PAMS74}
H.~A. Levine and L.~E. Payne.
\newblock Some nonexistence theorems for initial-boundary value problems with
  nonlinear boundary constraints.
\newblock {\em Proc. Am. Math. Soc.}, 46(2):277--284, 1974.

\bibitem{Levine-MMAS87}
H.~A. Levine, R.~A. Smith, and L.~E. Payne.
\newblock A potential well theory for the heat equation with a nonlinear
  boundary condition.
\newblock {\em Math. Method Appl. Sci.}, 9(1):127--136, 1987.

\bibitem{Vitillaro-PRES-05}
E.~Vitillaro.
\newblock Global existence for the heat equation with nonlinear dynamical
  boundary conditions.
\newblock {\em P. Roy. Soc. Edinb. A}, 135(1):175--207, 2005.

\bibitem{Vitillaro-2017ARMA}
E.~Vitillaro.
\newblock On the wave equation with hyperbolic dynamical boundary conditions,
  interior and boundary damping and source.
\newblock {\em Arch. Ration. Mech. Anal.}, 223(3):1183--1237, 2017.

\bibitem{Vitillaro-JDE18}
E.~Vitillaro.
\newblock On the wave equation with hyperbolic dynamical boundary conditions,
  interior and boundary damping and supercritical sources.
\newblock {\em J. Differ. Equ.}, 265(10):4873--4941, 2018.

\bibitem{Vitillaro-DCDS21}
E.~Vitillaro.
\newblock Blow–up for the wave equation with hyperbolic dynamical boundary
  conditions, interior and boundary nonlinear damping and sources.
\newblock {\em Discrete Contin. Dyn. Syst.-Ser. S}, 14(12):4575--4608, 2021.

\bibitem{Enzo-jde11}
J.~L. Vázquez and E.~Vitillaro.
\newblock Heat equation with dynamical boundary conditions of
  reactive–diffusive type.
\newblock {\em J. Differ. Equ.}, 250(4):2143--2161, 2011.

\bibitem{WillemMinimaxBook}
M.~Willem.
\newblock {\em {Minimax Theorems}}.
\newblock Progress in Nonlinear Differential Equations and Their Applications.
  Birkh\"{a}user Boston Inc, Boston, 1996.

\bibitem{YangXin-JDE16}
X.~Yang and Z.~Zhou.
\newblock Blow-up problems for the heat equation with a local nonlinear
  {Neumann} boundary condition.
\newblock {\em J. Differ. Equ.}, 261(5):2738--2783, 2016.

\end{thebibliography}
\end{document}